\documentclass[11pt,letterpaper]{amsart}

\usepackage{graphicx}
\usepackage{amssymb}
\usepackage{amsthm}
\usepackage{amsmath}
\usepackage{stmaryrd}
\usepackage{cite}
\usepackage[mathscr]{eucal}
\usepackage{hyperref}
\usepackage[margin=1.7in]{geometry}
\usepackage{comment}
\usepackage{caption}
\usepackage{subcaption}
\usepackage[permil]{overpic}
\usepackage{url}
\usepackage{placeins}

\title[A vanishing result for harmonic $(n-1,1)$-forms]{A vanishing result for harmonic $(n-1,1)$-forms}
\author[Matthias Wink]{Matthias Wink}
\address{Department of Mathematics, University of California, Santa Barbara, South Hall 6607, Santa Barbara, CA 93106, USA}
\email{wink@math.ucsb.edu}


\makeatletter
\@namedef{subjclassname@2020}{
\textup{2020} Mathematics Subject Classification}
\makeatother

\begin{document}
\newcommand{\Ext}{\bigwedge\nolimits}
\newcommand{\Div}{\operatorname{div}}
\newcommand{\Hol} {\operatorname{Hol}}
\newcommand{\diam} {\operatorname{diam}}
\newcommand{\Scal} {\operatorname{Scal}}
\newcommand{\scal} {\operatorname{scal}}
\newcommand{\Ric} {\operatorname{Ric}}
\newcommand{\Hess} {\operatorname{Hess}}
\newcommand{\grad} {\operatorname{grad}}
\newcommand{\Sect} {\operatorname{Sect}}
\newcommand{\Rm} {\operatorname{Rm}}
\newcommand{ \Rmzero } {\mathring{\Rm}}
\newcommand{\Rc} {\operatorname{Rc}}
\newcommand{\Curv} {S_{B}^{2}\left( \mathfrak{so}(n) \right) }
\newcommand{ \tr } {\operatorname{tr}}
\newcommand{ \id } {\operatorname{id}}
\newcommand{ \Riczero } {\mathring{\Ric}}
\newcommand{ \ad } {\operatorname{ad}}
\newcommand{ \Ad } {\operatorname{Ad}}
\newcommand{ \dist } {\operatorname{dist}}
\newcommand{ \rank } {\operatorname{rank}}
\newcommand{\Vol}{\operatorname{Vol}}
\newcommand{\dVol}{\operatorname{dVol}}
\newcommand{ \zitieren }[1]{ \hspace{-3mm} \cite{#1}}
\newcommand{ \pr }{\operatorname{pr}}
\newcommand{\diag}{\operatorname{diag}}
\newcommand{\Lagr}{\mathcal{L}}
\newcommand{\av}{\operatorname{av}}
\newcommand{ \floor }[1]{ \lfloor #1 \rfloor }
\newcommand{ \ceil }[1]{ \lceil #1 \rceil }
\newcommand{\Sym} {\operatorname{Sym}}
\newcommand{\bcirc}{ \ \bar{\circ} \ }
\newcommand{\conj}[1]{ \overline{ #1 } }
\newcommand{\sign}[1]{\operatorname{sign}(#1)}
\newcommand{\cone}{\operatorname{cone}}
\newcommand{\pbd}{\varphi_{bar}^{\delta}}

\newtheorem{theorem}{Theorem}[section]
\newtheorem{definition}[theorem]{Definition}
\newtheorem{example}[theorem]{Example}
\newtheorem{remark}[theorem]{Remark}
\newtheorem{lemma}[theorem]{Lemma}
\newtheorem{proposition}[theorem]{Proposition}
\newtheorem{corollary}[theorem]{Corollary}
\newtheorem{assumption}[theorem]{Assumption}
\newtheorem{acknowledgment}[theorem]{Acknowledgment}
\newtheorem{DefAndLemma}[theorem]{Definition and lemma}

\newcommand{\R}{\mathbb{R}}
\newcommand{\N}{\mathbb{N}}
\newcommand{\Z}{\mathbb{Z}}
\newcommand{\Q}{\mathbb{Q}}
\newcommand{\C}{\mathbb{C}}
\newcommand{\F}{\mathbb{F}}
\newcommand{\X}{\mathcal{X}}
\newcommand{\D}{\mathcal{D}}
\newcommand{\Cont}{\mathcal{C}}

\renewcommand{\labelenumi}{(\alph{enumi})}
\newtheorem{maintheorem}{Theorem}[]
\renewcommand*{\themaintheorem}{\Alph{maintheorem}}
\newtheorem*{theorem*}{Theorem}
\newtheorem*{corollary*}{Corollary}
\newtheorem*{remark*}{Remark}
\newtheorem*{example*}{Example}
\newtheorem*{question*}{Question}
\newtheorem*{definition*}{Definition}
\newtheorem{conjecture}{Conjecture}
\renewcommand*{\theconjecture}{\Alph{conjecture}}
\newtheorem*{conjecture*}{Conjecture}

\begin{abstract}
We prove that every primitive harmonic $(n-1,1)$-form of a closed K\"ahler manifold of complex dimension $n$ vanishes provided its K\"ahler curvature operator is $\frac{n^2-n+2}{2}$-positive if $n \geq 4$, respectively $\frac{7}{2}$-positive if $n=3$. As a consequence, a closed K\"ahler manifold of complex dimension $n \geq 3$ with $3$-positive K\"ahler curvature operator is a real cohomology complex projective space.
\end{abstract}

\maketitle

\vspace{-2mm}

\section*{Introduction}

The interaction of the curvature and the topology of closed K\"ahler manifolds is a fundamental question in complex geometry. Hodge theory provides a tool to obtain vanishing results for Hodge numbers by studying harmonic forms. In \cite{BochnerVectorFieldsAndRic}, Bochner proved that every holomorphic $p$-form vanishes provided the Ricci curvature is $k$-positive, that is, if the sum of the lowest $k$ eigenvalues of the Ricci curvature is positive, then $h^{p,0}=0$ for $k \leq p \leq n.$ Recently, many more vanishing results for holomorphic forms have been obtained in \cite{YangRCpositivity, NiZhengComparisonAndVanishingKaehler,NiFundamentalGroupRationalConnectedness, NiZhengPositivityAndKodaira}.

With regard to the second Betti number, the Bochner technique was used in \cite{BishopGoldbergCohomKaehler, WilkingALieAlgebraicApproach} to prove that $b_2(M) =1$ provided $M$ has positive (orthogonal) bisectional curvature. In fact, due to the work of Mori \cite{MoriProjectiveManifoldsWithAmpleTangentBundles},  Siu-Yau \cite{SiuYauCompactKaehlerPosBiholCurv}, Chen \cite{ChenPosOrthBisectCurv} and Gu-Zhang \cite{GuZhangExtensionMokTheorem}, closed K\"ahler manifolds with positive orthogonal bisectional curvature are biholomorphic to complex projective space. 

Vanishing results for all Hodge numbers, i.e., characterizations of K\"ahler manifolds as a real cohomology projective space, were moreover established recently in \cite{PetersenWinkHodgeNumbers,WangYangWeitzenbockBochnerKodaira,BNPSWCalabiCurvatureOperator}. 

To explain these results, recall that the Riemannian curvature operator of a K\"ahler manifold vanishes on the orthogonal complement of the holonomy algebra. The induced operator $\mathcal{R} \colon \Lambda^{1,1}_{\R}TM \to \Lambda^{1,1}_{\R}TM$ is the K\"ahler curvature operator. We denote its eigenvalues by $\lambda_1 \leq \ldots \leq \lambda_{n^2}$ and call it $k$-positive if $\lambda_1 + \ldots + \lambda_{\floor{k}} + (k - \floor{k}) \lambda_{\floor{k}+1} >0.$

In \cite{PetersenWinkHodgeNumbers}, Petersen and the author proved that a closed K\"ahler manifold with $\left(3-\frac{2}{n}\right)$-positive curvature operator is a real cohomology $\mathbb{CP}^n.$ This is a consequence of a vanishing result for individual Hodge numbers $h^{p,q}.$ However, while vanishing of primitive $(p,p)$-forms only required $(n-p+1)$-positive K\"ahler curvature operator, the methods in \cite{PetersenWinkHodgeNumbers} required curvature conditions much stronger than $p$-positive Ricci curvature to prove $h^{p,0}=0.$ The assumption of $\left(3-\frac{2}{n}\right)$-positivity is exactly the curvature assumption in \cite[Theorem B]{PetersenWinkHodgeNumbers} needed to control primitive $(n-1,1)$-forms and force $h^{n-1,1}=0$. These observations raised the question of finding curvature conditions that imply vanishing of primitive harmonic $(p,q)$-forms that naturally recover Bochner's result for $(p,0)$-forms. 

To this end, in \cite{BNPSWCalabiCurvatureOperator}, the authors consider the Calabi curvature operator $C \colon T^{1,0}M \odot T^{1,0}M \to T^{1,0}M \odot T^{1,0}M$ and prove that a closed K\"ahler manifold is a real cohomology $\mathbb{CP}^n$ provided $C$ is $\frac{n}{2}$-positive. For even $n$ this is optimal in the sense that the complex quadric with its symmetric metric has $\frac{n}{2}$-nonnegative Calabi curvature operator but $b_n=2.$ However, the original question of providing refined curvature conditions for the K\"ahler curvature operator remained open. 

The goal of the paper is to introduce a new curvature condition that implies $h^{n-1,1}=0.$ 

\begin{maintheorem}
\label{TheoremHodgeNumbers}
    Let $(M,g)$ be a closed K\"ahler manifold of complex dimension $n.$ If $n \geq 4$ and the K\"ahler curvature operator is $\frac{n^2-n+2}{2}$-positive, then $h^{n-1,1}(M)=0.$

    If $n=3$ and the K\"ahler curvature operator is $3$-positive, then $h^{2,1}(M)=0.$
\end{maintheorem}

We note that while Bochner's work \cite{BochnerVectorFieldsAndRic} implies $h^{p,0}=0$ provided the K\"ahler curvature operator is $np$-positive, the curvature condition in \cite{PetersenWinkHodgeNumbers} to control primitive $(p,q)$-forms is $\left( n+1 - \frac{p^2+q^2}{p+q} \right)$-positivity, and hence is (at most) linear in $n$. In contrast, the number of eigenvalues needed in Theorem \ref{TheoremHodgeNumbers} is quadratic in $n,$ similar to Bochner's result. In fact, the methods introduced in this paper can be used to recover that $h^{p,0}=0$ for manifolds with $np$-positive K\"ahler curvature operator, \cite{BochnerVectorFieldsAndRic}.

Combining Theorem \ref{TheoremHodgeNumbers} with \cite{BochnerVectorFieldsAndRic, PetersenWinkHodgeNumbers}, we obtain the following generalization of \cite[Theorem A]{PetersenWinkHodgeNumbers}.

\begin{maintheorem}
\label{MainTheoremCPn}
        Let $(M,g)$ be a closed K\"ahler manifold of complex dimension $n \geq 3$ and let 
        \begin{align*}
            k_n = \begin{cases}
                \hspace{6mm} 3 & \text{ if } n=3,4, \\
                \hspace{5mm} \frac{17}{5} & \text{ if } n=5, \\
                4-\frac{2}{n-1} & \text{ if } n \geq 6.
            \end{cases}
        \end{align*}
        
        If the K\"ahler curvature operator of $(M,g)$ is $k_n$-positive, then $M$ has the real cohomology ring of complex projective space. 
\end{maintheorem}

\begin{remark*}
    \label{RemarkEpsilonImprovement}
    \normalfont
    For $n \geq 4,$ we explain in Example \ref{PWMaximizer} how to obtain $\varepsilon_n>0$ such that Theorem \ref{MainTheoremCPn} holds if the K\"ahler curvature operator is $(k_n+ \varepsilon_n)$-positive, see also Conjecture \ref{CPnConjecture} below.
\end{remark*}

The proofs of Theorem \ref{TheoremHodgeNumbers} and Theorem \ref{MainTheoremCPn} rely on a new understanding of the Bochner formula for primitive $(p,q)$-forms. It is therefore natural that they have rigidity and estimation analogues.

Rigidity or estimation results specifically for K\"ahler manifolds are also established in or follow from \cite{BochnerVectorFieldsAndRic, HowardSmythWuNonnegBiSect, WuCompactKaehlerNonnegBiSectII,MokUniformizationKaehler, GuNewProofGeneralizedFrankelConj, GuZhangExtensionMokTheorem, PetersenWinkHodgeNumbers,BNPSWCalabiCurvatureOperator, WangYangWeitzenbockBochnerKodaira}. For example, Mok \cite{MokUniformizationKaehler} studies K\"ahler manifolds with nonnegative bisectional curvature. 

Let $P^{p,q}M = \ker \left( \Lambda \colon \Lambda^{p,q}T^*M \to \Lambda^{p-1,q-1}T^*M \right)$ be the vector bundle of primitive $(p,q)$-forms, where $\Lambda$ is the dual of the Lefschetz map. The underlying result behind Theorems \ref{TheoremHodgeNumbers} and \ref{MainTheoremCPn} is

\begin{maintheorem}
\label{PrimitiveFormTheorem}
    Let $(M,g)$ be a closed K\"ahler manifold of complex dimension $n \geq 3$ and let $\mathcal{R} \colon \Lambda^{1,1}_{\R}TM \to \Lambda^{1,1}_{\R}TM$ denote its K\"ahler curvature operator with eigenvalues $\lambda_1 \leq \ldots \leq \lambda_{n^2}.$ 
        Let 
        \begin{align*}
            l_n = 
            \begin{cases}
                \frac{n^2-n+2}{2}=1+\binom{n}{2} & \text{ if } n \geq 4, \\
                \hspace{13.5mm} \frac{7}{2} & \text{ if } n = 3.
            \end{cases}
        \end{align*}
    \begin{enumerate}
        \item If $\mathcal{R}$ is $l_n$-positive, then every primitive harmonic $(n-1,1)$-form vanishes.
        \item If $\mathcal{R}$ is $l_n$-nonnegative, then every primitive harmonic $(n-1,1)$-form is parallel.
        \item Let $\kappa \leq 0$ and $D>0.$ There exists $C(n,\kappa D^2)>0$ such that the following holds. 
        
        If $\lambda_1 + \ldots + \lambda_{\floor{l_n}} + ( l_n - \floor{l_n}) \lambda_{\floor{l_n}+1} \geq l_n \kappa,$ $\Ric \geq (2n-1)\kappa g$ and $\operatorname{diam}(M)<D,$ then the vector space $\mathcal{PH}^{n-1,1}(M)$ of primitive harmonic $(n-1,1)$-forms has dimension at most
        \begin{align*}
        \hspace{8mm} \dim \mathcal{PH}^{n-1,1}(M) \leq \rank P^{n-1,1}M \cdot \exp  \left( C(n, \kappa D^2) \sqrt{-\kappa D^2 } \right).
        \end{align*}
    \end{enumerate}
\end{maintheorem}

We note that the Ricci curvature is bounded from below by the sum of the lowest $n$ eigenvalues of the K\"ahler curvature operator. To apply the techniques of P. Li \cite{LiSobolevConstant} and Gallot \cite{GallotSobolevEstimates} to obtain estimation results, a lower bound on the Ricci curvature is therefore assumed explicitly. \vspace{2mm}

We now explain the key idea for the proof of Theorems \ref{TheoremHodgeNumbers} -- \ref{PrimitiveFormTheorem}. Every harmonic form $\varphi \in \Lambda^{p,q}T^*M$ satisfies the Bochner formula
\begin{align*}
    \Delta \frac{1}{2} | \varphi |^2 = | \nabla \varphi |^2 + g( \Ric_L( \varphi), \overline{\varphi}).
\end{align*}
Thus, if $g( \Ric_L( \varphi), \overline{\varphi}) \geq 0,$ then the maximum principle implies that $\varphi$ is parallel. Following the ideas in \cite{WinkThorpeTrickForBochner}, we define $\mathcal{A}_{\varphi} \colon \Lambda^{1,1}_{\R}TM \to \Lambda^{1,1}_{\R}TM$ by 
\begin{align*}
    g(\mathcal{A}_{\varphi}(\alpha), \beta) = \operatorname{Re} g( \alpha \varphi, \overline{\beta \varphi}),
\end{align*}
where $\alpha \varphi, \beta \varphi$ denote the Lie algebra action of $\alpha, \beta \in \mathfrak{u}(TM) \cong \Lambda^{1,1}_{\R}TM$ on $\Lambda^{p,q}T^*M.$ 

The first key observation is that  
\begin{align*}
    g( \Ric_L(\varphi), \overline{\varphi}) = \tr( \mathcal{R} \mathcal{A}_{\varphi} ) = \tr( \mathcal{R} ( \mathcal{A}_{\varphi} + \hat{\eta}) ) 
\end{align*}
for every $\eta \in \Lambda^{2,2}_{\R}T^*M$ as the space of K\"ahler curvature operators is orthogonal to real $(2,2)$-forms. This allows us to establish in Corollary \ref{CharacterizationBoundsOnRicL} that $g( \Ric_L(\varphi), \overline{\varphi}) \geq 0$ on manifolds with $k$-nonnegative K\"ahler curvature operators provided that for every primitive $\varphi \in \Lambda^{p,q}T^*M$ there exists $\eta \in \Lambda^{2,2}_{\R}T^*M$ such that 
        \begin{align*}
        0 \leq \mathcal{A}_{\varphi} + \hat{\eta} 
        \leq
        \frac{\tr( \mathcal{A}_{\varphi}+\hat{\eta})}{k}  \id_{\Lambda^{1,1}_{\R}TM}.
        \end{align*}
        
The second key insight is that $\tr_{\Lambda^{1,1}_{\R}TM}( \hat{\eta})$ may be positive. This is different from the Riemannian case considered in \cite{WinkThorpeTrickForBochner}, where $\tr_{\Lambda^2TM}(\hat{\eta})=0$ for every $4$-form $\eta.$ This is the insight that allows us to prove vanishing results that are quadratic in the number of eigenvalues involved in Theorem \ref{TheoremHodgeNumbers}, since $\tr(\mathcal{A}_{\varphi}) = (3n-2)|\varphi|^2$ for primitive $(n-1,1)$-forms, see Proposition \ref{PropertiesActionOperator}.

In particular, for $n \geq 4$ we show in Theorem \ref{AdjustmentActionOperator}  that for every primitive $(n-1,1)$-form $\varphi$ there exists $\eta \in \Lambda^{2,2}_{\R}T^*M$ such that 
    \begin{align*}
        0 \leq \mathcal{A}_{\varphi} + \hat{\eta} \leq 2 | \varphi |^2 \id_{\Lambda^{1,1}_{\R}TM} \ \text{ and } \ \tr( \mathcal{A}_{\varphi} + \hat{\eta}) = (n^2-n+2)| \varphi|^2.
    \end{align*}
Thus, Corollary \ref{CharacterizationBoundsOnRicL} implies that $g( \Ric_L(\varphi), \overline{\varphi}) \geq 0$ if the K\"ahler curvature operator is $\frac{n^2-n+2}{2}$-nonnegative and in this case every primitive harmonic $(n-1,1)$-form is parallel. 

Example \ref{ExampleSharpCandidate} shows that the above choice of $\eta$ is optimal. Based on these examples, we propose the following conjectural curvature conditions for vanishing of primitive $(p,q)$-forms.

\begin{conjecture}
    \label{ProposedCurvatureConditions}
    Let $p, q \geq 1$ with $p+ q \leq n.$ 
    
    If the K\"ahler curvature operator of a closed complex $n$-dimensional K\"ahler manifold is $\frac{1}{2}\left((p+q)(n+1)-2pq-\delta_{|p-q|, 1} \right)$-positive, then every primitive harmonic $(p,q)$-form vanishes. 
\end{conjecture}

%Note that $\delta_{|p-q|,1} =1$ if and only if $|p-q|=1.$
Theorem \ref{PrimitiveFormTheorem} covers Conjecture \ref{ProposedCurvatureConditions} for $(p,q)=(n-1,1)$ and \cite[Theorem B]{PetersenWinkHodgeNumbers} confirms it in the case of primitive $(1,1)$-forms. A positive resolution of Conjecture \ref{ProposedCurvatureConditions} would imply

\begin{conjecture}
\label{CPnConjecture}
    If $(M,g)$ is a closed K\"ahler manifold of complex dimension $n$ with $n$-positive K\"ahler curvature operator, then $M$ is a real cohomology $\mathbb{CP}^n.$
\end{conjecture}

Note that Theorem \ref{MainTheoremCPn} covers the case $n=3.$ \vspace{2mm}

\textbf{AI disclosure.} The author used OpenAI’s ChatGPT as an assistive tool in the development of the paper. The author takes full responsibility for all mathematical arguments and the manuscript was entirely written by the author. \vspace{2mm}

\textbf{Acknowledgments.} I thank Fabio Ricci for comments on a previous version of the paper, in particular with regard to Example \ref{ExampleSharpCandidate}. 

The author's research is partially supported by travel grant SFI-MPS-TSM-00026032 funded by the Simons Foundation International and administered by the Simons Foundation.

\section{Preliminaries}

Let $(V,g)$ be a Euclidean vector space. We identify $\Lambda^2V$ with $\mathfrak{so}(V)$ via
\begin{align*}
    (x \wedge y)z = g(x,z)y - g(y,z)x
\end{align*}
and consider the inner product $\left\langle A, B \right\rangle = - \frac{1}{2} \tr (AB)$ on $\mathfrak{so}(V).$ Accordingly, if $\{ e_i \}$ is an orthonormal basis for $V$, then $\{ e_{i_1} \wedge \ldots \wedge e_{i_p} \}_{1 \leq i_1 < \ldots < i_p \leq \dim V}$ is an orthonormal basis for $\Lambda^pV.$

A subgroup $G$ of the orthogonal group $O(V)$ acts on $\Lambda^pV$ by
\begin{align*}
    g \cdot x_1 \wedge \ldots \wedge x_p = gx_1 \wedge \ldots \wedge gx_p.
\end{align*}
The induced action on $p$-forms $\Lambda^p V^{*}$ is
\begin{align*}
    (g \cdot \varphi)(X_1, \ldots, X_p) = \varphi(g^{-1}X_1, \ldots, g^{-1}X_p)
\end{align*}
and the corresponding Lie algebra action of $\mathfrak{g} \subset \mathfrak{so}(V)$ is 
\begin{align*}
    (\Xi \varphi)(X_1, \ldots, X_p) = - \sum_{i=1}^p \varphi(X_1, \ldots, \Xi X_i, \ldots, X_p).
\end{align*}

From now on, let $(V,g,J)$ be a $2n$-dimensional Euclidean vector space with compatible almost complex structure $J.$ Recall that $V_{\C}=V \otimes_{\R} \C = V^{1,0} \oplus V^{0,1}$ and $\Lambda^{p,q} V_\C^* = \Lambda^p (V^{1,0})^* \otimes \Lambda^q(V^{0,1})^*.$ The real form
\begin{align*}
    \Lambda^{p,p}_\R V^* \subseteq  \Lambda^{p,p} V_\C^* 
\end{align*}
consists of all $(p,p)$-forms invariant under complex conjugation. Note that the identification $\Lambda^2V \cong \mathfrak{so}(V)$ induces the isomorphism
\begin{align*}
    \Lambda^{1,1}_\R V^* \cong \mathfrak{u}(V) = \left\lbrace L \in \mathfrak{so}(V) \ \vert \ L  \circ J = J \circ L, \ g(L \cdot, \cdot ) + g( \cdot, L \cdot ) = 0 \right\rbrace.
\end{align*}

It is shown in \cite{AlekseevskiRiemannExceptionalHolonomy}, \cite[Proposition 2.1]{NagyConnectionsWithTorsion} or \cite[Section 1.2]{AVCGHamiltonian2Forms} that the vector space of self-adjoint operators on $\mathfrak{u}(V)$ has the orthogonal $U(n)$-invariant decomposition
\begin{align}
\label{OrthogonalDecompositionSymUV}
    \Sym^2(\mathfrak{u}(V)) = \Sym_B^2(\mathfrak{u}(V)) \oplus \widehat{\Lambda^{2,2}_\R V^*},
\end{align}
where $\Sym_B^2({\mathfrak{u}(V)}) = \ker \left( L_1 \otimes L_2 + L_2 \otimes L_1 \mapsto L_1 \wedge L_2 \right)$ is the space of algebraic K\"ahler curvature operators and $\hat{\eta} \colon \Lambda^{1,1}_{\R} V \to \Lambda^{1,1}_{\R}V$ is the operator induced by $\eta \in \Lambda^{2,2}_\R V^*$ via $g(\hat{\eta}(x \wedge y), z \wedge w) = \eta(x,y,z,w).$

    \begin{remark}
    \normalfont
For an algebraic curvature tensor $R(X,Y,Z,W)$ we define the induced algebraic curvature operator by
\begin{align*}
    \mathcal{R} \colon \Lambda^2V \to \Lambda^2V, \ g(\mathcal{R}(X \wedge Y), Z \wedge W) = R(X,Y,Z,W).
\end{align*}

If $R$ is K\"ahler, $R(JX,JY,Z,W)=R(X,Y,JZ,JW)=R(X,Y,Z,W),$ then the orthogonal complement of $\Lambda^{1,1}_{\R}V \subseteq \Lambda^2V$ is in the kernel of $\mathcal{R}$ and $\mathcal{R}_{|\mathfrak{u}(V)} \colon \mathfrak{u}(V) \to \mathfrak{u}(V).$ We extend $R$ complex multi-linearly and obtain a self-adjoint operator $\mathcal{R} \colon \mathfrak{u}_{\C}(V) \to  \mathfrak{u}_{\C}(V).$ The map
\begin{align*}
    &\mathfrak{u} (V) \to \mathfrak{u}_{\C} (V) \cong \mathfrak{gl}_{\C}(V^{1,0}) \\
    & \hspace{5mm} L  \mapsto L_{\C} \mapsto  (L_{\C})_{|V^{1,0}}
\end{align*}
induces the first isomorphism in
\begin{align*}
    \mathfrak{u}(V) & \cong \{ L \in \operatorname{End}(V^{1,0}) \ | \ g(L(X), \overline Y) = - g(X, \overline{L(Y)}) \} \\
    & \cong  \{ L \in \operatorname{End}(V^{1,0}) \ | \ g(L(X), \overline Y) = + g(X, \overline{L(Y)}) \} 
\end{align*}
and the second isomorphism is given by $L \mapsto \sqrt{-1}L.$ With the isomorphisms $\mathfrak{gl}_{\C}(V^{1,0}) \cong \Lambda^{1,1}V,$ $X \otimes g( \cdot , \overline{Y}) \mapsto - X \wedge \overline{Y}$ and $\mathfrak{gl}_{\C}(V^{1,0}) = V^{1,0} \otimes (V^{1,0})^*$ we obtain the decomposition
\begin{align*}
    \Sym^2(\Lambda^{1,1}V) & \cong \Sym^2(\mathfrak{u}_{\C}(V)) \cong \Sym^2(\mathfrak{gl}_{\C}(V^{1,0})) \\
    & \cong \mathfrak{gl}_{\C}( \Sym^2(V^{1,0})) \oplus \mathfrak{gl}_{\C}( \Lambda^2 V^{1,0}) \\
    & \cong \mathfrak{gl}_{\C}( V^{1,0} \odot V^{1,0}) \oplus \Lambda^{2,2} V
\end{align*}
where $X \odot Y = X \otimes Y + Y \otimes X.$ By taking the real form, algebraic K\"ahler curvature operators correspond to self-adjoint endomorphisms on $ V^{1,0} \odot V^{1,0},$ see also \cite{CalabiVesentiniCompactLocallySymmetricKaehler, Sitaramayya}. Under these isomorphisms, $\mathcal{R}$ corresponds to the Calabi curvature operator $C \colon  V^{1,0} \odot V^{1,0} \to  V^{1,0} \odot V^{1,0}$ defined by
\begin{align*}
    g( \mathcal{R}(X \wedge \overline{Y}), \overline{Z \wedge \overline{W}}) = R(X,\overline{Y},\overline{Z}, W) = \frac{1}{4} g( C( X \odot W), \overline{Y} \odot \overline{Z} ).
\end{align*}
%Note that the conventions for the curvature tensor in this paper agree with the conventions in \cite{BNPSWCalabiCurvatureOperator} whereas \cite{ElHLiNienhausPetersenStanfieldWinkCalabiCPn} uses the other convention.
\end{remark}

\section{The Lichnerowicz Laplacian on $(p,q)$-forms}

For a closed K\"ahler manifold $(M,g)$ the Hodge Laplacian agrees with the Lichnerowicz Laplacian $\Delta = \nabla^*\nabla + \Ric_L$ and preserves the $U(n)$-irreducible modules of $\Lambda^{p,q}T^*M.$ Moreover, every harmonic form $\varphi$ satisfies
\begin{align*}
    \Delta \frac{1}{2} | \varphi |^2 = | \nabla \varphi |^2 + g( \Ric_L(\varphi), \overline{\varphi}).
\end{align*}
The basic idea of the Bochner technique is the observation that, by the maximum principle, $\varphi$ is parallel provided $g(\Ric_L(\varphi),\overline{\varphi}) \geq 0.$ Here, the curvature term $\Ric_L$ can be computed in terms of the K\"ahler curvature operator $\mathcal{R} \colon \Lambda^{1,1}_{\R}TM \to \Lambda^{1,1}_{\R}TM$. Specifically, if $\{ \Xi_{\alpha} \}$ is an orthonormal eigenbasis for $\mathcal{R}$ with corresponding eigenvalues $\{ \lambda_{\alpha} \},$ then
\begin{align}
\label{CurvatureTermViaKaehlerCurvatureOperator}
    g( \Ric_L(\varphi), \overline{\varphi}) = \sum_{\alpha} \lambda_{\alpha} | \Xi_{\alpha} \varphi |^2,
\end{align}
see \cite[Proposition 1.6]{PetersenWinkHodgeNumbers}. In particular, the curvature term of the Lichnerowicz Laplacian depends on the Lie algebra action of the holonomy algebra $\mathfrak{u}(n)$ on $(p,q)$-forms. Following the ideas in \cite{WinkThorpeTrickForBochner}, we define

\begin{definition}
\label{DefinitionActionOperator}
    Let $(V,g,J)$ be a $2n$-dimensional Euclidean vector space with compatible almost complex structure. The {\em action operator} $\mathcal{A}_{\varphi} \colon \Lambda^{1,1}_{\R}V \to \Lambda^{1,1}_{\R} V$ associated to $\varphi \in \Lambda^{p,q}V_{\C}^*$ is defined by
\begin{align*}
    g(\mathcal{A}_{\varphi}(\alpha), \beta) = \operatorname{Re} g( \alpha \varphi, \overline{\beta \varphi})
\end{align*}
for all $\alpha, \beta \in \Lambda^{1,1}_{\R}V \cong \mathfrak{u}(V).$ 
\end{definition}

Let $\omega( \cdot , \cdot ) = g( J \cdot, \cdot)$ denote the K\"ahler form, $L\varphi = \omega \wedge \varphi$ be the Lefschetz operator and let $\Lambda$ denote its adjoint. A $(p,q)$-form $\varphi$ is called {\em primitive} if $\Lambda \varphi = 0.$ 

The action operator has the following basic properties.

\begin{proposition}
\label{PropertiesActionOperator}
\text{}%The following hold.
\begin{enumerate}
    \item The action operator associated to $\varphi \in \Lambda^{p,q}V^*$ is self-adjoint, nonnegative and satisfies
\begin{align*} 
    \mathcal{A}_{U \cdot \varphi} (\alpha)= U \cdot \mathcal{A}_{\varphi}\left(U^{-1} \cdot \alpha \right)
\end{align*}
for all $U \in U(V)$ and $\alpha \in \Lambda^{1,1}_{\R}V.$
\item $\mathcal{A}_{e^{it} \varphi} = \mathcal{A}_{\varphi}$ for all $\varphi \in \Lambda^{p,q}V^*$ and $t \in \R.$
\item If $\varphi \in \Lambda^{p,q}V^*$ is primitive, then 
\begin{align*}
    \tr( \mathcal{A}_{\varphi} ) = (2pq+(p+q)(n+1-(p+q)))|\varphi|^2.
\end{align*}
\end{enumerate}
\end{proposition}
\begin{proof} (a), (b) follow from the definition of the action operator and (c) is computed in \cite[Proposition 3.2]{PetersenWinkHodgeNumbers}.
\end{proof}

The key tool for obtaining estimates on the curvature term of the Lichnerowicz Laplacian is the following basic observation.

\begin{proposition}
\label{CurvatureTermBochnerCharacterization}
If $\varphi \in \Lambda^{p,q}V^*$ and  $\mathcal{R} \in \Sym^2_B(\mathfrak{u}(V)),$ then
\begin{align*}
     g( \Ric_L(\varphi), \overline{\varphi}) = \tr( \mathcal{R} \mathcal{A}_{\varphi}) = \tr (\mathcal{R} (\mathcal{A}_{\varphi} + \hat{\eta}))
    \end{align*}
for every $\eta \in \Lambda^{2,2}_{\R} V^{*}.$
\end{proposition}
\begin{proof}
    The first equality is exactly \eqref{CurvatureTermViaKaehlerCurvatureOperator} and orthogonality of the decomposition in \eqref{OrthogonalDecompositionSymUV} says that $\tr( \mathcal{R} \hat{\eta} ) = 0$ for every $\mathcal{R} \in \Sym^2_B(\mathfrak{u}(V))$ and $\eta \in \Lambda^4V^*.$
\end{proof}

As a consequence of Proposition \ref{CurvatureTermBochnerCharacterization} and \cite[Theorem 2.7]{WinkThorpeTrickForBochner} we obtain 

\begin{corollary}
\label{CharacterizationBoundsOnRicL}
    The following are equivalent:
    \begin{enumerate}
        \item Every $k$-nonnegative $\mathcal{R} \in \Sym^2_B(\mathfrak{u}(V))$  satisfies $g(\Ric_L(\varphi),\overline{\varphi}) \geq 0$ for all primitive $\varphi \in \Lambda^{p,q}V^*.$
        \item For every primitive $\varphi \in \Lambda^{p,q}V^*$ there exists $\eta \in \Lambda^{2,2}_{\R} V^{*}$ such that 
        \begin{align*}
        0 \leq \mathcal{A}_{\varphi} + \hat{\eta} 
        \leq
        \frac{\tr( \mathcal{A}_{\varphi}+\hat{\eta})}{k} \id_{\Lambda^{1,1}_{\R}V}.
        \end{align*}
    \end{enumerate}
\end{corollary}

\begin{remark}
\label{RemarkPositivityEstimation}
    \normalfont
    The key computation that shows (b) implies (a) is 
\begin{align*}
    g(\Ric_L( \varphi), \overline{\varphi}) = \tr( \mathcal{R}(\mathcal{A}_{\varphi} + \hat{\eta})) \geq \frac{\tr(\mathcal{A}_{\varphi}+ \hat{\eta})}{k} \left( \sum_{\alpha=1}^{\floor{k}} \lambda_{\alpha} + (k-\floor{k}) \lambda_{\floor{k}+1} \right),
\end{align*}
    see \cite[Lemma 2.5 (a)]{WinkThorpeTrickForBochner}. 
    
    Suppose that condition (b) in Corollary \ref{CharacterizationBoundsOnRicL} is satisfied and $\eta=\eta_{\varphi}$ can be chosen so that $\tr( \mathcal{A}_{\varphi} + \widehat{\eta_{\varphi}})=c(p,q,n)|\varphi|^2$ for some $c(p,q,n)>0$ and every primitive $\varphi \in \Lambda^{p,q}V^*.$ Then the following hold.
    \begin{enumerate}
        \item If $\mathcal{R} \in \Sym^2_B(\mathfrak{u}(V))$ is $k$-positive, then $g(\Ric_L(\varphi), \overline{\varphi}) > 0$ for every primitive $(p,q)$-form $\varphi \neq 0.$
        \item If $\mathcal{R} \in \Sym^2_B(\mathfrak{u}(V))$ with $\lambda_1 + \ldots + \lambda_{\floor{k}} + (k-\floor{k}) \lambda_{\floor{k}+1} \geq k \kappa,$ then $g(\Ric_L(\varphi),\overline{\varphi}) \geq \kappa c(p,q,n)|\varphi|^2$.
    \end{enumerate}

\end{remark}

\begin{remark}
    \normalfont
    The estimate $| L \varphi |^2 \leq (p+q) |L|^2|\varphi|^2$ in \cite[Proposition 3.4]{PetersenWinkHodgeNumbers} for primitive $(p,q)$-forms $\varphi$ shows that condition (b) is satisfied with $\eta=0$ and $k=\frac{2pq+(p+q)(n+1-(p+q))}{p+q}.$ In particular, for primitive $(n-1,1)$-forms, $g(\Ric_L(\varphi), \overline{\varphi}) \geq 0$ provided the K\"ahler curvature operator is $\left( 3-\frac{2}{n} \right)$-nonnegative.
\end{remark}

\section{The action operator on $(n-1,1)$-forms}

In this section we study the action operator for primitive $(n-1,1)$-forms. This relies on a specific normal form. The idea is that the primitive part of 
    \begin{align*}
        \Lambda^{n-1,1}V^* & \cong \det( (V^{1,0})^*) \otimes V^{1,0} \otimes V^{1,0} \\
        & \cong \det( (V^{1,0})^*) \otimes \left( \Sym^2(V^{1,0}) \oplus \Lambda^2V^{1,0} \right)
    \end{align*}
is exactly $\det( (V^{1,0})^*) \otimes  \Sym^2(V^{1,0})$ and complex symmetric matrices can be diagonalized. 

Let $Z_1, \ldots, Z_n$ be a unitary basis for $V^{1,0}$, let $Z^j = g( \cdot, \overline{Z_j})$ be the dual basis and note that $\overline{Z^j} = g( Z_j, \cdot )$. The complex volume form is given by $\Omega=Z^1 \wedge \ldots \wedge Z^n \in \det( (V^{1,0})^*).$ 

\begin{lemma}
\label{NormalForm}
    For every primitive form $\varphi \in \Lambda^{n-1,1}V^*$ there exists a unitary basis $Z_1, \ldots, Z_n$ for $V^{1,0},$ $t \in \R,$ and $0 \leq a_1 \leq \ldots \leq a_n$ such that 
    \begin{align*}
        \varphi = e^{it} \sum_{k=1}^n a_k  \iota_{Z_k} \Omega \wedge \overline{Z^k} = (-1)^{n-1} e^{it} \sum_{k=1}^n a_k  Z^1 \wedge \ldots \wedge \overline{Z^k} \wedge \ldots \wedge Z^n.
    \end{align*}
\end{lemma}
\begin{proof}
    Let $Z_1, \ldots, Z_n$ be a unitary basis for $V^{1,0}.$ Note that
    \begin{align*}
        \iota_{Z_j}\Omega = (-1)^{j-1} Z^1 \wedge \ldots \wedge \widehat{Z^j} \wedge \ldots \wedge Z^n
    \end{align*}
    and thus every $\varphi \in \Lambda^{n-1,1}V^*$ can be written as 
    \begin{align*}
        \varphi = \sum\nolimits_{i,j=1}^n a_{ij} \iota_{Z_i} \Omega \wedge \overline{Z^j}
    \end{align*}
    Recall that the dual of the Lefschetz operator is given by $\Lambda = - \sqrt{-1} \sum \iota_{\overline{Z_k}}\iota_{Z_k}.$ It follows that 
    \begin{align*}
        \Lambda \varphi = \sqrt{-1}  \sum\nolimits_{i < j} (-1)^{i+j+n} (a_{ij} - a_{ji}) Z^1 \wedge  \ldots \wedge \widehat{Z^i} \wedge \ldots \wedge \widehat{Z^j} \wedge \ldots \wedge Z^n.
    \end{align*}
    As $Z^1 \wedge  \ldots \wedge \widehat{Z^i} \wedge \ldots \wedge \widehat{Z^j} \wedge \ldots \wedge Z^n$ are linearly independent, $\varphi$ is primitive if and only if $a_{ij}=a_{ji}$ for all $1\leq i,j \leq n.$

    We now want to diagonalize $A = (a_{ij})$ via a unitary transformation. Let $Z'_j = \sum_{i=1}^n u_{ij} Z_i$ for $U = (u_{ij}) \in U(n).$ Then $Z'^{j}= \sum _{i=1}^n\overline{u_{ij}} Z^i,$ $\Omega{'} = \det(U)^{-1} \Omega$ and $\overline{Z^{'j}}= \sum_{i=1}^n u_{ij} \overline{Z^i}.$ It follows that $A = \det(U)^{-1} U A' U^T.$ There is $U \in U(n)$ such that $A=UDU^T$ with $D=\diag(a_1, \ldots, a_n)$ and $0 \leq a_1 \leq \ldots \leq a_n.$ In the associated new coordinates, $\varphi$ is in normal form with $e^{it}=\det(U).$
\end{proof}

We now determine the matrix representation of $\mathcal{A}_{\varphi}$ for a $\varphi$ in normal form. To this end, let $e_1, \ldots, e_n, f_1=Je_1, \ldots, f_n=Je_n$ be an adapted orthonormal basis for $V.$ Then, as in \cite[Section 1.2]{PetersenWinkHodgeNumbers}, 
\begin{align*}
R_{ij} & = \frac{1}{\sqrt{2}} ( e_i \wedge e_j + f_i \wedge f_j ) \ \text{ for }  1 \leq i < j \leq n, \\
I_{ij} & = \frac{1}{\sqrt{2}} ( e_i \wedge f_j + e_j \wedge f_i ) \ \text{ for }  1 \leq i < j \leq n,  \\
I_{ii} & =  e_i \wedge f_i \ \hspace{24.7mm}\text{ for }  1 \leq i \leq n
\end{align*}
is an orthonormal basis for $\Lambda^{1,1}_{\R}V \cong \mathfrak{u}(V).$ Note that $R_{ij}=-R_{ji}$ and $I_{ij}=I_{ji}.$

\begin{proposition}
\label{ActionOperatorOnNormalForm}
    Let $Z_i = \frac{1}{\sqrt{2}}( e_i - \sqrt{-1} f_i)$ and for $a_1, \ldots, a_n \in \R$ define $\varphi_0 = \sum_{k=1}^n a_k  \iota_{Z_k} \Omega \wedge \overline{Z^k} = (-1)^{n-1} \sum_{k=1}^n a_k  Z^1 \wedge \ldots \wedge \overline{Z^k} \wedge \ldots \wedge Z^n.$
Then 
\begin{align*}
    \mathcal{A}_{\varphi_0}(R_{ij}) & = (a_i - a_j)^2 R_{ij}, \\
    \mathcal{A}_{\varphi_0}(I_{ij}) & = (a_i + a_j)^2 I_{ij}, \\
    g(\mathcal{A}_{\varphi_0}(I_{ii}),I_{jj}) & = \begin{cases}
        |\varphi_0|^2 - 2(a_i^2+a_j^2) & \text{ if } i \neq j, \\
        |\varphi_0|^2  & \text{ if } i = j.
    \end{cases}
\end{align*}
\end{proposition}
\begin{proof}
From the definition of the action it is straightforward to compute that
    \begin{align*}
        R_{ij} Z_k & = \frac{1}{\sqrt{2}} \left( \delta_{ik} Z_j - \delta_{jk} Z_i \right), \hspace{5.65mm} R_{ij} \overline{Z_k}  = \frac{1}{\sqrt{2}} \left( \delta_{ik} \overline{Z_j} - \delta_{jk} \overline{Z_i} \right),\\
        I_{ij} Z_k & = \frac{\sqrt{-1}}{\sqrt{2}} \left( \delta_{ik} Z_j + \delta_{jk} Z_i \right), \hspace{4mm} I_{ij} \overline{Z_k}  = - \frac{\sqrt{-1}}{\sqrt{2}} \left( \delta_{ik} \overline{Z_j} + \delta_{jk} \overline{Z_i} \right),  \\
        I_{ii} Z_k & = \sqrt{-1} \delta_{ik} Z_i, \hspace{23.25mm} I_{ii} \overline{Z_k}  = - \sqrt{-1} \delta_{ik} \overline{Z_i}.
    \end{align*}
Since $(L Z^i)(\overline{Z_j})= - Z^i(L\overline{Z_j})=0$ and $(L \overline{Z^i})(Z_j)= - \overline{Z^i}(LZ_j)=0$ for all $L \in \mathfrak{u}(V),$ we obtain similarly 
    \begin{align*}
        R_{ij} Z^k & = \frac{1}{\sqrt{2}} \left( \delta_{ik} Z^j - \delta_{jk} Z^i \right), \hspace{8.8mm} R_{ij} \overline{Z^k}  = \frac{1}{\sqrt{2}} \left( \delta_{ik} \overline{Z^j} - \delta_{jk} \overline{Z^i} \right),\\
        I_{ij} Z^k & = - \frac{\sqrt{-1}}{\sqrt{2}} \left( \delta_{ik} Z^j + \delta_{jk} Z^i \right), \hspace{4mm} I_{ij} \overline{Z^k}  = \frac{\sqrt{-1}}{\sqrt{2}} \left( \delta_{ik} \overline{Z^j} + \delta_{jk} \overline{Z^i} \right),  \\
        I_{ii} Z^k & = - \sqrt{-1} \delta_{ik} Z^i, \hspace{24.2mm} I_{ii} \overline{Z^k}  = \sqrt{-1} \delta_{ik} \overline{Z^i}.
    \end{align*}
It follows that 
\begin{align*}
    R_{ij} Z^1 \wedge \ldots \wedge \overline{Z^k} \wedge \ldots \wedge Z^n & = \frac{\delta_{ik}-\delta_{jk}}{\sqrt{2}} \left(  Z^1 \wedge \ldots \wedge \underset{\text{$j$-th slot}}{\overline{Z^i}} \wedge \ldots \wedge Z^n \right.  \\
    & \hspace{30mm}+ \left. Z^1 \wedge \ldots \wedge \underset{\text{$i$-th slot}}{\overline{Z^j}} \wedge \ldots \wedge Z^n \right) \\
    & = (-1)^{n-1} \frac{\delta_{ik}-\delta_{jk}}{\sqrt{2}} \left( \iota_{Z_i} \Omega \wedge \overline{Z^j} +  \iota_{Z_j} \Omega \wedge \overline{Z^i} \right), \\
    I_{ij} Z^1 \wedge \ldots \wedge \overline{Z^k} \wedge \ldots \wedge Z^n & = \sqrt{-1}  \ (-1)^{n-1} \frac{\delta_{ik}+\delta_{jk}}{\sqrt{2}} \left( \iota_{Z_i} \Omega \wedge \overline{Z^j} +  \iota_{Z_j} \Omega \wedge \overline{Z^i} \right), \\
    I_{ii} Z^1 \wedge \ldots \wedge \overline{Z^k} \wedge \ldots \wedge Z^n & = \sqrt{-1} \  (2 \delta_{ik} -1) Z^1 \wedge \ldots \wedge \overline{Z^k} \wedge \ldots \wedge Z^n \\
    & = - \sqrt{-1} \ (-1)^{n-1}  (1 - 2 \delta_{ik}) \iota_{Z_k} \Omega \wedge \overline{Z^k}.
\end{align*}
Therefore, 
\begin{align*}
    R_{ij} \varphi_0 
    & = \frac{a_i-a_j}{\sqrt{2}} \left( \iota_{Z_i} \Omega \wedge \overline{Z^j} +  \iota_{Z_j} \Omega \wedge \overline{Z^i} \right), \\
    I_{ij} \varphi_0 & = \sqrt{-1}  \ \frac{a_i+a_j}{\sqrt{2}} \left( \iota_{Z_i} \Omega \wedge \overline{Z^j} +  \iota_{Z_j} \Omega \wedge \overline{Z^i} \right), \\
    I_{ii} \varphi_0
    & = - \sqrt{-1} \ \left( \varphi_0 - 2 a_i \iota_{Z_i} \Omega \wedge \overline{Z^i} \right).
\end{align*}
Direct inspection shows that except for $I_{ii}, I_{jj}$ all of $R_{ij} \varphi_0, I_{ij} \varphi_0, I_{ii} \varphi_0$ are mutually orthogonal for $i<j.$ Definition \ref{DefinitionActionOperator} of $\mathcal{A}_{\varphi_0}$ yields the claim.
\end{proof}

\begin{theorem}
\label{AdjustmentActionOperator}
    Let $n \geq 4.$ For every primitive $(n-1,1)$-form $\varphi \in \Lambda^{n-1,1}V^*$ there exists $\eta \in \Lambda^{2,2}_{\R}V^*$ such that 
    \begin{align*}
        0 \leq \mathcal{A}_{\varphi} + \hat{\eta} \leq 2 | \varphi |^2 \id_{\Lambda^{1,1}_{\R}V} \ \text{ and } \ \tr( \mathcal{A}_{\varphi} + \hat{\eta}) = (n^2-n+2)| \varphi|^2.
    \end{align*}
\end{theorem}
\begin{proof}
    Pick an adapted orthonormal basis $e_1, \ldots, e_n,$ $f_1=Je_1, \ldots, f_n=Je_n$ for $V$ such that $\varphi = e^{it} \varphi_0$ as in Lemma \ref{NormalForm}. As $\mathcal{A}_{\varphi}=\mathcal{A}_{\varphi_0}$ due to Proposition \ref{PropertiesActionOperator} (b) and $\mathcal{A}_{\varphi_0}$ has been computed explicitly in Proposition \ref{ActionOperatorOnNormalForm}, the goal is to construct $\eta \in \Lambda^{2,2}_{\R}V^*$ in these coordinates accordingly. 

    To this end, observe that
    \begin{align*}
        e_i \wedge f_i \wedge e_j \wedge f_j = I_{ii} \wedge I_{jj} = - R_{ij} \wedge R_{ij} = - I_{ij} \wedge I_{ij}.
    \end{align*}
    Thus, by setting $\eta_{ij}=e^i \wedge f^i \wedge e^j \wedge f^j$ to be the dual, we obtain from the evaluation pairing $g( \hat{\eta}(x \wedge y), z \wedge w)=\eta(x,y,z,w)=\eta(x \wedge y \wedge z \wedge w)$ the identities
    \begin{align*}
        \widehat{\eta_{ij}}(R_{kl}) & = - \delta_{ik}\delta_{jl}  R_{ij}, \\
        \widehat{\eta_{ij}}(I_{kl}) & = - \delta_{ik}\delta_{jl}  I_{ij}, \\
        g(\widehat{\eta_{ij}}(I_{kk}),I_{ll}) & =\begin{cases}
            1 & \text{ if } \{i,j\}=\{k,l\}, \\ 
            0 & \text{ otherwise.}
        \end{cases}
    \end{align*}
    Define  
    \begin{align*}
        \eta = \sum_{i<j} ( 2 (a_i^2 +a_j^2) - |\varphi_0|^2 - h_{ij}) \eta_{ij}
    \end{align*}
    with auxiliary parameters $h_{ij} \in \R$ to be determined later. The above identities and Proposition \ref{ActionOperatorOnNormalForm} show that
    \begin{align*}
        (\mathcal{A}_{\varphi_0} + \hat{\eta})(R_{ij})&=(-(a_i+a_j)^2+|\varphi_0|^2+h_{ij}) R_{ij}, \\
        (\mathcal{A}_{\varphi_0} + \hat{\eta})(I_{ij})&=( -(a_i-a_j)^2 +  |\varphi_0|^2 + h_{ij}) I_{ij},\\
        g((\mathcal{A}_{\varphi_0} + \hat{\eta})(I_{ii}),I_{jj}) &= \begin{cases}
            - h_{ij} & \text{ if } i \neq j, \\
            |\varphi_0|^2 & \text{ if } i=j.
        \end{cases}
    \end{align*}
    To find conditions on $h_{ij}$ to obtain $0 \leq \mathcal{A}_{\varphi} + \hat{\eta} \leq 2 | \varphi |^2 \id_{\Lambda^{1,1}_{\R}V},$ note that if
    \begin{align*}
    A = \begin{pmatrix}
   |\varphi_0|^2 & -h_{12} &-h_{13} &\cdots& -h_{1n}\\
   -h_{12} & |\varphi_0|^2 & -h_{23}&\cdots& -h_{2n}\\
   -h_{13}& -h_{23}& |\varphi_0|^2 &\cdots& -h_{3n}\\
    \vdots&\vdots&\vdots&\ddots&\vdots\\
   -h_{1n}& -h_{2n}& -h_{3n} &\cdots& |\varphi_0|^2
    \end{pmatrix}
    \end{align*}
    with $h_{ij} \geq 0,$ then 
    \begin{align*}
        x^TAx & = \sum_{1 \leq i < j \leq n} h_{ij}(x_i - x_j)^2 + \sum_{i=1}^n \left( | \varphi_0|^2 - \sum_{j:j \neq i} h_{ij} \right) x_i^2 \\
        & \leq 2 \sum_{i=1}^n x_i^2  \sum_{j:j \neq i} h_{ij} + |\varphi_0|^2 |x|^2  - \sum_{i=1}^n x_i^2 \sum_{j : j \neq i} h_{ij} \\
        & = |\varphi_0|^2 |x|^2 + \sum_{i=1}^n x_i^2 \sum_{j : j \neq i} h_{ij}.
    \end{align*}
    Thus, if 
    \begin{align}
    \label{NonnegativityConditions}
        \max \{ (a_i+a_j)^2-| \varphi_0|^2, 0 \} \leq h_{ij} \ \text{ and } \
 \sum_{j : j \neq i} h_{ij} = |\varphi_0|^2 \ \text{ for all } i,
    \end{align}
    then $\mathcal{A}_{\varphi_0} + \hat{\eta}$ is nonnegative, and as $h_{ij} \leq | \varphi|^2,$ also $\mathcal{A}_{\varphi_0}+ \hat{\eta} \leq 2 |\varphi_0|^2 \id_{\Lambda^{1,1}_{\R}V}.$ According to Proposition \ref{ExistenceOfCoefficients}, as $|\varphi_0|^2 = \sum_{i=1}^n a_i^2,$ we may indeed pick $h_{ij}$ according to \eqref{NonnegativityConditions} for $n \geq 4.$
    
    By construction then we have $2 \sum_{i < j} h_{ij} = \sum_{i=1}^n \sum_{j: j \neq i} h_{ij} = n| \varphi_0|^2$ and 
    \begin{align*}
        \tr( \mathcal{A}_{\varphi_0} + \hat{\eta}) & = 2 \sum_{i <j} \left( |\varphi_0|^2+h_{ij}-a_i^2-a_j^2 \right) + n |\varphi_0|^2 \\
        & = (n(n-1)+n- 2(n-1)+n)  | \varphi_0|^2 
        = (n^2-n+2)|\varphi_0|^2
    \end{align*}
    as claimed. 
\end{proof}

\begin{proposition}
    \label{ExistenceOfCoefficients} Let $n=2$ or $n \geq 4.$ For all $a_1, \ldots, a_n \geq 0$  there exist $h_{ij}=h_{ji} \in \R$ for $1 \leq i < j \leq n$ such that
    \begin{align*}
                \max \{ (a_i+a_j)^2-\sum_{k=1}^n a_k^2, 0 \} \leq h_{ij} \text{ and } \sum_{j:j \neq i} h_{ij} =\sum_{k=1}^n a_k^2 \ \text{ for all } i.
    \end{align*}
\end{proposition}
\begin{proof}
    We may assume $n \geq 4.$ The $h_{ij}$ exist if and only if there are $x_{ij}$ such that 
    \begin{align*}
        x_{ij} & =  h_{ij} - \max \{ (a_i+a_j)^2-\sum_{k=1}^n a_k^2, 0 \} \geq 0 \ \text{ and } \\\sum_{j \neq i} x_{ij} & = d_i := \sum_{k=1}^n a_k^2 - \sum_{j:j \neq i} \max \{ (a_i+a_j)^2-\sum_{k=1}^n a_k^2, 0 \}.
    \end{align*}
    
    For the vertices $V=\{ 1, \ldots, n \}$ let $x_{ij}$ be the weight on the edge $\{ i, j\}.$ The weighted degree of $i \in V$ is by definition $\deg_w(i)=\sum_{j:j \neq i} x_{ij}.$ Due to \cite[Theorem 6]{HakimiRealizabilityAsDegreesOfGraphs}, 
    nonnegative numbers $d_1 \leq \ldots \leq d_n$ can be realized as the weighted degrees of the graph, $d_i = \sum_{j:j \neq i} x_{ij},$  if and only if $2d_n \leq \sum_{i=1}^n d_i.$

    For completeness, we note that the 'if'-implication also follows from Farkas' lemma, \cite{FarkasEinfacheUngleichungen,KuhnTuckerNonlinearProgramming}. If $e_1, \ldots, e_n$ denotes the standard basis of $\R^n,$ then either there exist $x_{ij} \geq 0$ such that $\sum_{i<j} x_{ij} (e_i +e_j)=d:=\sum_{i=1}^n d_i e_i$ or there exists $y \in \R^n$ such that $\left\langle y,e_i+e_j\right\rangle \geq 0$ and $\left\langle y,d\right\rangle<0.$ If $y_i + y_j \geq 0,$ there is at most one coordinate $y_{i_0}<0.$ Thus, $\left\langle y,d\right\rangle \geq d_{i_0}y_{i_0} - y_{i_0} \sum_{i:i \neq i_0} d_i = y_{i_0} \left( 2d_{i_0} - \sum_{i=1}^n d_i \right) \geq 0$ and Farkas' lemma applies.  

    It remains to check that $d_i \geq 0$ and $2 \max \{ d_i \} \leq \sum_{i=1}^n d_i$ in the initial problem.

    Let $|A|^2=\sum_{i=1}^n a_i^2$ and consider $J_i = \{ j \neq i \ | \ (a_i+a_j)^2 > |A|^2 \}$. Note that $\sum_{j \in J_i} a_j \leq \left( |J_i| \sum_{j \in J_i} a_j^2 \right)^{1/2} \leq |J_i|^{1/2} \left( |A|^2 -a_i^2 \right)^{1/2}$ and therefore
    \begin{align*}
        |A|^2 - d_i & = \sum_{j \in J_i} \left( (a_i+a_j)^2 - |A|^2 \right) \\
        & = - |J_i|(|A|^2-a_i^2) + 2 a_i \sum_{j \in J_i} a_j + \sum_{j \in J_i} a_j^2 \\
        & \leq - \left( |J_i|^{1/2} \left( |A|^2 - a_i^2 \right)^{1/2} - a_i \right)^{2} + |A|^2.
    \end{align*}
    This shows $d_i \geq 0.$

    To prove $2\max\{ d_i \} \leq \sum_{i=1}^n d_i$ we first make an observation. There are no distinct indices $i,j,k,l$ such that both $|A|^2 < (a_i+a_j)^2$ and $|A|^2<(a_k+a_l)^2$ as otherwise $2|A|^2 < (a_i+a_j)^2 + (a_k+a_l)^2 \leq 2|A|^2.$ Thus, either there is $1 \leq k \leq n$ such that
    \begin{align*}
        \sum_{i < j} \max \{ (a_i+a_j)^2-|A|^2, 0 \} = \sum_{j \in J_k} \left( (a_k+a_j)^2-|A|^2 \right) \leq |A|^2
    \end{align*}
    as established above, or there are distinct $i,j,k$ such that 
    \begin{align*}
        \sum_{i < j} \max \{ (a_i+a_j)^2-|A|^2, 0 \} & = (a_i + a_j)^2 + (a_i +a_k)^2 + (a_j +a_k)^2 - 3 |A|^2 \\
        & \leq 4(a_i^2+a_j^2+a_k^2) - 3|A|^2 \leq |A|^2.
    \end{align*}
    It follows that
    \begin{align*}
        \sum_{i=1}^n d_i & = n |A|^2 -2 \sum_{i<j} \max \{ (a_i+a_j)^2-|A|^2, 0 \} \\
        & \geq (n-2)|A|^2 \geq (n-2) \max \{ d_i \} \geq 2 \max \{ d_i \}
    \end{align*}
    since $n \geq 4.$
\end{proof}

\begin{remark}
\normalfont
    Proposition \ref{ExistenceOfCoefficients} does not hold for $n=3.$ Consider for example $(a_1,a_2,a_3)=(0,1,1).$
\end{remark}

\begin{remark}
\label{DimensionThreeAdjustment}
    \normalfont
    If we remove the requirement  $\sum_{j:j \neq i} h_{ij} = |\varphi_0|^2$ for $i=1, \ldots, n$ from \eqref{NonnegativityConditions} and define $h_{ij} = \max \{ (a_i+a_j)^2 - |\varphi_0|^2, 0 \}$ in the proof of Theorem \ref{AdjustmentActionOperator}, then by continuing the proof verbatim one obtains $\mathcal{A}_{\varphi_0} + \hat{\eta} \geq 0$ and  
    \begin{align*}
        \tr(\mathcal{A}_{\varphi_0} + \hat{\eta})& =(n^2-2n+2)|\varphi_0|^2 + 2 \sum\nolimits_{i<j} h_{ij}, \\
        \lambda_{\max}(\mathcal{A}_{\varphi_0} + \hat{\eta}) & \leq  |\varphi_0|^2 + \max_i \sum\nolimits_{j:j \neq i} h_{ij} \leq |\varphi_0|^2 + \sum\nolimits_{i<j} h_{ij},
    \end{align*}
    where $\lambda_{\max}$ is the largest eigenvalue. 
    
    Set $h=\sum_{i<j} h_{ij}$ and note that $h \leq |\varphi_0|^2.$ This follows as in the proof of Proposition \ref{ExistenceOfCoefficients}, where it is phrased as $\sum_{i<j} \max \{ (a_i+a_j)^2-|A|^2,0\} \leq |A|^2.$ The quotient $\frac{N+2h}{1+h}$ is decreasing in $h$ and hence
    \begin{align*}
        \frac{\tr(\mathcal{A}_{\varphi_0} + \hat{\eta})}{\lambda_{\max}(\mathcal{A}_{\varphi_0}+\hat{\eta})} 
        \geq \frac{(n^2-2n+2)|\varphi_0|^2+2h}{|\varphi_0|^2+h} 
        \geq \frac{n^2-2n+4}{2}.
    \end{align*}
    For $n=3$ we obtain
    \begin{align}
    \label{EqDimThreeAdjustment}
        0 \leq \mathcal{A}_{\varphi_0} + \hat{\eta} \leq \frac{2}{7} \tr(\mathcal{A}_{\varphi_0} + \hat{\eta}) \id_{\Lambda^{1,1}_{\R}V}.
    \end{align}
\end{remark}

\textit{Proof of Theorem \ref{PrimitiveFormTheorem}.} Let $\varphi \in \Lambda^{n-1,1}T^*M$ be a primitive harmonic $(n-1,1)$-form. If $n \geq 4,$ then Theorem \ref{AdjustmentActionOperator} shows that there exists $\eta \in \Lambda^{2,2}_{\R}T^*M$ such that 
\begin{align*}
    0 \leq \mathcal{A}_{\varphi} + \hat{\eta} \leq 2 | \varphi |^2 \id_{\Lambda^{1,1}_{\R}TM} \ \text{ and } \ \tr( \mathcal{A}_{\varphi} + \hat{\eta}) = (n^2-n+2)| \varphi|^2.
\end{align*}
If the K\"ahler curvature operator $\mathcal{R}$ is $\frac{n^2-n+2}{2}$-nonnegative, then Corollary \ref{CharacterizationBoundsOnRicL} shows that 
\begin{align*}
    g( \Ric_L( \varphi), \overline{\varphi}) \geq 0.
\end{align*}
As every harmonic form satisfies the Bochner formula 
\begin{align*}
    \Delta \frac{1}{2} | \varphi |^2 = | \nabla \varphi |^2 + g( \Ric_L( \varphi), \overline{\varphi}),
\end{align*}
the maximum principle shows that $\varphi$ is parallel. This proves the rigidity case in Theorem \ref{PrimitiveFormTheorem} (b). 

Moreover, for the vanishing result in Theorem \ref{PrimitiveFormTheorem} (a), note that if the K\"ahler curvature operator is $\frac{n^2-n+2}{2}$-positive, then by Remark \ref{RemarkPositivityEstimation} (a) it follows that $g( \Ric_L( \varphi), \overline{\varphi}) > 0$ unless $\varphi=0.$ 

For the estimation theorem, Theorem \ref{PrimitiveFormTheorem} (c), recall that by Remark \ref{RemarkPositivityEstimation}, the assumption $\lambda_1 + \ldots + \lambda_{\floor{l_n}} + ( l_n - \floor{l_n}) \lambda_{\floor{l_n}+1} \geq l_n \kappa$ implies that 
\begin{align*}
    g( \Ric_L( \varphi), \overline{\varphi}) \geq (n^2-n+2) \kappa |\varphi|^2= 2 l_n \kappa |\varphi|^2
\end{align*}
and by assumption the Ricci curvature is bounded from below. Hence, the work of P. Li \cite{LiSobolevConstant} and Gallot \cite{GallotSobolevEstimates} applies and proves the estimation theorem, see also \cite[Theorem 1.9]{PetersenWinkNewCurvatureConditionsBochner}. 

If $n=3$, then replace Theorem \ref{AdjustmentActionOperator} by Remark \ref{DimensionThreeAdjustment} and note that in this case
\begin{align*}
    0 \leq \mathcal{A}_{\varphi_0} + \hat{\eta} \leq \frac{2}{7} \tr(\mathcal{A}_{\varphi_0} + \hat{\eta}) \id_{\Lambda^{1,1}_{\R}TM}
\end{align*}
according to \eqref{EqDimThreeAdjustment}. The rest of the proof remains unchanged. $\hfill \Box$\vspace{2mm}

\textit{Proof of Theorem \ref{TheoremHodgeNumbers}.} By the Hard Lefschetz Theorem, see \cite[Proposition 3.3.13]{HuybrechtsComplexGeometry}, it suffices to show that every primitive harmonic $(n-1,1)$-form and every harmonic $(n-2,0)$-form vanishes. Theorem \ref{PrimitiveFormTheorem} (a) says that indeed every primitive harmonic $(n-1,1)$-form vanishes. Moreover, as the sum of the lowest $p$ eigenvalues of the Ricci curvature is bounded below by the sum of the lowest $np$ eigenvalues of the K\"ahler curvature operator, the Ricci curvature is $(n-2)$-positive and hence Bochner's theorem \cite{BochnerVectorFieldsAndRic} proves that every harmonic $(n-2,0)$-form vanishes. $\hfill \Box$\vspace{2mm}

\textit{Proof of Theorem \ref{MainTheoremCPn}.} Bochner's work \cite{BochnerVectorFieldsAndRic} shows that all holomorphic forms vanish, $h^{p,0}=0$ for $1 \leq p \leq n.$ Theorem \ref{TheoremHodgeNumbers} shows $h^{n-1,1}(M)=0$.

By the Hard Lefschetz Theorem, it suffices to prove vanishing of primitive harmonic $(p,q)$-forms in the remaining irreducible modules. A sufficient condition is given in  \cite[Theorem B]{PetersenWinkHodgeNumbers}. Specifically, if $\mathcal{R}$ denotes the K\"ahler curvature operator, then primitive harmonic 
\begin{enumerate}
    \item $(n-2,2)$-forms vanish provided $\mathcal{R}$ is $\left( 5 - \frac{8}{n} \right)$-positive,
    \item $(n-2,1)$-forms vanish provided $\mathcal{R}$ is $\left( 4 - \frac{2}{n-1} \right)$-positive,
    \item $(1,1)$-forms vanish provided $\mathcal{R}$ is $n$-positive,
\end{enumerate}
and all primitive harmonic forms in other irreducible modules vanish under weaker curvature assumptions. The curvature condition given in Theorem \ref{MainTheoremCPn} is exactly the minimum of these conditions. Note that in the special case $n=3$ we have $h^{2,1}(M)=0$ by Theorem \ref{TheoremHodgeNumbers}. $\hfill \Box$

\section{Examples}

Example \ref{ExampleSharpCandidate} motivates the curvature conditions in Conjecture \ref{ProposedCurvatureConditions} and Example \ref{PWMaximizer} explains that for $n \geq 4$ there is $\varepsilon_n>0$ such that Theorem \ref{MainTheoremCPn} holds if the K\"ahler curvature operator is $(k_n+ \varepsilon_n)$-positive.

\begin{example}
    \label{ExampleSharpCandidate}
    \normalfont Let $Z_1, \ldots, Z_n$ be a unitary basis for $V^{1,0}.$ For $1 \leq p \leq q$ with $p+q \leq n$ define 
    \begin{align*}
        \varphi = 2^{-p/2} \bigwedge_{j=1}^p \left( Z^{2j-1} \wedge \overline{Z^{2j-1}} - Z^{2j} \wedge \overline{Z^{2j}} \right) \wedge \overline{Z^{2p+1}} \wedge \ldots \wedge \overline{Z^{p+q}}.
    \end{align*}
    The forms $Z^{2j-1} \wedge \overline{Z^{2j-1}} - Z^{2j} \wedge \overline{Z^{2j}}$ are primitive and one concludes that $\varphi$ is primitive with $|\varphi|=1$. 

    The spectral decomposition of $\mathcal{A}_{\varphi}$ is given by
       \begin{align*}
        \mathcal{A}_{\varphi} = \begin{cases}
             q-p & \text{on } \R L \text{ if } q>p, \text{ where } L=\frac{1}{\sqrt{q-p}} \left( I_{2p+1, 2p+1} + \ldots + I_{p+q,p+q} \right), \\
            \hspace{3mm} 2  & \text{on } \operatorname{span} \{ R_{2j-1,2j}, I_{2j-1,2j} \}  \text{ for } 1 \leq j \leq p, \\
            \hspace{3mm} \frac{1}{2}  & \text{on } \operatorname{span} \{ R_{ij}, I_{ij} \} \text{ for } 1 \leq i < j \leq 2p, \ (i,j) \neq (2k-1,2k), \\
            \hspace{3mm}   & \hspace{29.04mm} \text{ for } 1 \leq i \leq 2p, \ 2p+1 \leq j \leq n, \\
            \hspace{3mm}   & \hspace{29.04mm} \text{ for } 2p+1 \leq i \leq p+q, \ p+q+1 \leq j \leq n, \\
            \hspace{3.4mm} 0 & \text{otherwise.}
        \end{cases}
    \end{align*}
    Note that $\tr( \mathcal{A}_{\varphi} )=(p+q)(n+1)-2pq-(p-q)^2.$ Moreover, we claim that 
        \begin{align}
        \label{OptimalQuotient}
        \sup_{\substack{\eta \in \Lambda^{2,2}_{\R}V^* \\ \mathcal{A}_{\varphi} + \hat{\eta} \geq 0}} \frac{\tr( \mathcal{A}_{\varphi} + \hat{\eta})}{\lambda_{\max} ( \mathcal{A}_{\varphi} + \hat{\eta})} = \frac{(p+q)(n+1)-2pq- \delta_{|p-q|,1}}{2},
    \end{align}
    which suggests to consider the curvature conditions in Conjecture \ref{ProposedCurvatureConditions}.% Note that $\delta_{|p-q|,1} =1$ if and only if $|p-q|=1.$

    To prove \eqref{OptimalQuotient}, we make a preliminary observation. If $J = \sum_{i=1}^n I_{ii},$ then 
    \begin{align*}
        \tr(\hat{\eta}) & = \sum_{i < j} \left( g( \hat{\eta}(R_{ij}), R_{ij}) + g( \hat{\eta}(I_{ij}), I_{ij}) \right) \\
        & =  - \sum_{i,j} g(\hat{\eta}(I_{ii}),I_{jj}) = - g( \hat{\eta}(J),J).
    \end{align*}
    Moreover, 
    \begin{align*}
        \tr( \mathcal{A}_{\varphi} + \hat{\eta} ) + g( (\mathcal{A}_{\varphi} + \hat{\eta})(J),J)= (p+q)(n+1)-2pq.
    \end{align*}
    Thus, if $\mathcal{A}_{\varphi} + \hat{\eta} \geq 0,$ then $\tr( \mathcal{A}_{\varphi} + \hat{\eta} ) \leq (p+q)(n+1)-2pq.$

    We now consider different cases. If $p=q,$ then $0 \leq \mathcal{A}_{\varphi} \leq 2 \id_{\Lambda^{1,1}_{\R}V}.$ Moreover, if $\mathcal{A}_{\varphi} + \hat{\eta} \geq 0$ for some  $\eta \in \Lambda^{2,2}_{\R}V^*,$ then $\hat{\eta}(I_{ii})=0$ for all $i$ and hence $\tr (\hat{\eta}) =0$ by the above observation. Therefore, any such $\hat{\eta}$ does not change $\tr( \mathcal{A}_{\varphi} + \hat{\eta})$ and $\lambda_{\max}(\mathcal{A}_{\varphi} + \hat{\eta}) \geq g((\mathcal{A}_{\varphi} + \hat{\eta})(R_{2j-1,2j}), R_{2j-1,2j}) = 2.$

    If $|p-q|=1,$ then an analogous argument applies. 

    If $|p-q| \geq 2,$ set
    \begin{align*}
        \eta = - \frac{q-p}{q-p-1} \sum_{2p+1 \leq i < j \leq p+q} e^i \wedge f^i \wedge e^j \wedge f^j.
    \end{align*}
    It follows that $0 \leq \mathcal{A}_{\varphi} + \hat{\eta} \leq 2 \id_{\Lambda^{1,1}_{\R}V}$ and $\tr( \mathcal{A}_{\varphi} + \hat{\eta}) = (p+q)(n+1)-2pq.$ To see that this gives the optimal quotient, suppose that $\mathcal{A}_{\varphi} + \hat{\eta} \geq 0$ for some $\eta \in \Lambda^{2,2}_{\R}V^*.$ As before, observe that $\hat{\eta}(I_{ii})=0$ for $1 \leq i \leq 2p$. Thus, $\lambda_{\max}(\mathcal{A}_{\varphi} + \hat{\eta}) \geq 2$ and the preliminary observation yields $\tr( \mathcal{A}_{\varphi} + \hat{\eta} ) \leq (p+q)(n+1)-2pq$ as claimed.
\end{example}

\begin{example}
    \label{PWMaximizer}    \normalfont
    For $p,q \geq 1$ with $3 \leq p+q \leq n,$ an $L \in \Lambda^{1,1}_{\R}V$ with $|L|=1$ and a primitive $(p,q)$-form $\varphi \in \Lambda^{p,q}V^*$ with $|\varphi|=1$ (and in addition $\varphi=\overline{\varphi}$ if $p=q$) maximizes the estimate $|L\varphi|^2 \leq (p+q)|L|^2 |\varphi|^2$ of \cite[Proposition 3.4]{PetersenWinkHodgeNumbers} if and only if (after a suitable choice of basis)
    \begin{align*}
        L= \pm \frac{1}{\sqrt{p+q}} \left( \sum_{i=1}^p I_{ii} - \sum_{i=p+1}^{p+q} I_{ii} \right)
    \end{align*}
    and  
    \begin{align*}
        \varphi = \begin{cases}
            Z^1 \wedge \ldots \wedge Z^p \wedge \overline{Z^{p+1}} \wedge \ldots \wedge \overline{Z^{p+q}} & \text{ if } p \neq q, \\
            \frac{1}{\sqrt{2}} ( Z^1 \wedge \ldots \wedge Z^p \wedge \overline{Z^{p+1}} \wedge \ldots \wedge \overline{Z^{2p}}  & \\
            \hspace{10mm} + (-1)^p Z^{p+1} \wedge \ldots \wedge Z^{2p} \wedge \overline{Z^{1}} \wedge \ldots \wedge \overline{Z^{p}} )   & \text{ if } p = q.
        \end{cases}
    \end{align*}
    The corresponding action operator satisfies
    \begin{align*}
        \mathcal{A}_{\varphi} = \begin{cases}
            \hspace{3.2mm} 1 & \text{ on } \operatorname{span} \{ R_{ij}, I_{ij} \} \ \text{ for } 1 \leq i \leq p, \ p+1 \leq j \leq p+q, \\
            \hspace{3mm} \frac{1}{2}  & \text{ on } \operatorname{span} \{ R_{ij}, I_{ij} \} \ \text{ for } 1 \leq i \leq p+q, \ p+q < j \leq n, \\
            p+q & \text{ on } \R L, \\
            \hspace{3.4mm} 0 & \text{ on the orthogonal complement.}
        \end{cases}
    \end{align*}
    Choosing $\eta = - \frac{p+q}{2} L \wedge L,$ it follows that $0 \leq \mathcal{A}_{\varphi} + \hat{\eta} \leq \id_{\Lambda^{1,1}_{\R}V}.$ With the same strategy as in the proof of \cite[Theorem 3.6]{WinkThorpeTrickForBochner} one can show that there is $\varepsilon(p,q)>0$ such that for every primitive $\varphi \in \Lambda^{p,q}V^*$ there exists $\eta \in \Lambda^{2,2}_{\R}V^*$ such that
    \begin{align*}
        0 \leq \mathcal{A}_{\varphi} + \hat{\eta} \leq (p+q-\varepsilon) |\varphi|^2 \id_{\Lambda^{1,1}_{\R}V}.
    \end{align*}
    In particular, if $\varphi_{\max}$ is a maximizer closest to $\varphi,$ then $\varphi_{\max}$ determines $L$ as above and one can choose $\eta = - \frac{p+q}{4} L \wedge L$. Otherwise $\eta=0$ works. 
\end{example}

%\bibliography{References}
%\bibliographystyle{amsalpha}

%\begin{comment}

%\end{comment}

\end{document}